\documentclass[12pt]{amsart}
\usepackage{epsfig, graphics, color}
\usepackage{amsmath,amsfonts,amssymb,amsthm,graphicx,verbatim,subfig}

\newcommand{\RR}{{\mathbb{R}}}

\newcommand{\CC}{{\mathbb{C}}}
  
\graphicspath{  {../Figures_paper2/}  }

\theoremstyle{plain}
\newtheorem{theorem}{Theorem}[section]

\newtheorem{lemma}[theorem]{Lemma}

\newtheorem{claim}[theorem]{Claim}

\author{Joel Hass and Patrice Koehl} 

\begin{document}
\title{A Metric for genus-zero surfaces}

\begin{abstract}
We present a new method  to compare the shapes of genus-zero surfaces. 
We introduce a measure of mutual stretching, the {\em symmetric distortion energy},
and establish the existence of  a conformal diffeomorphism between any two genus-zero surfaces that
minimizes this energy.
We then prove that the energies of the minimizing diffeomorphisms
give a metric on the space of genus-zero Riemannian surfaces. 
This metric and the corresponding optimal diffeomorphisms are shown to
have properties that are highly desirable for applications. 
\end{abstract}

\date{\today}

\maketitle

\section{Introduction}
The problem of comparing the  shapes of surface arises in many fields, including facial recognition, image processing, brain cortex analysis, protein structure analysis and computer vision.
It is referred to by names such as {\em surface registration, surface warping, best fit, shape analysis} and {\em geometric morphometrics}.
In this paper we introduce a new method to compare the shapes of two genus-zero surfaces. 
The method is based on a sequence of two energy minimizations, first minimizing the Dirichlet energy to produce a conformal map and then
minimizing a symmetric distortion energy,  defined in Section~\ref{align}.
It  produces a metric $d_{sd}$ on the space of piecewise-smooth surfaces genus-zero Riemannian surfaces, which we call the {\em symmetric distortion} metric.
In addition to giving a distance between any pair of genus-zero surfaces,
the method also produces an optimal correspondence  between them, a diffeomorphism whose
symmetric distortion energy defines the symmetric distortion distance.

A common approach to shape comparison of surfaces in $\RR^3$
 is to search for a Euclidean motion moving one surface close to the second, and to then measure in some way
the setwise difference between the two repositioned surfaces. 
Such approaches are {\em extrinsic}, as they consider not just the
two-dimensional geometry of the surface, but also the particular geometric embedding of the shape in
space. In extrinsic geometry, a hand in different configurations represents very different geometric shapes.
From the {\em intrinsic} point of view, which we use, a hand in different poses represents close to identical geometries. 
The intrinsic approach  has significant advantages when comparing surfaces
that can be flexible.

Our method of comparing two shapes involves finding an optimal diffeomorphism from one  to the other, a map that minimizes a measure of surface distortion.   This is often not the case  
in methods that compare surfaces by creating vectors of shape signatures based on features such as diameter, curvatures,
spectral properties, and spherical harmonics.
The existence of an explicit correspondence realizing the minimal distance is
very useful in many applications, and gives an advantage over methods, both extrinsic and intrinsic, that 
give distance measures without producing surface correspondences.

A key property of the measure of shape difference that we introduce is that it
gives a mathematical metric on the space of shapes of genus-zero surfaces.
A {\em metric} on a set $\mathcal X$ is a distance function $d:{\mathcal X} \times {\mathcal X} \to \RR$ that satisfies three properties:
\begin{enumerate}
\item   $d(S,T) \ge 0$, and $d(S,T) = 0$ if and only if $S = T$.
\item $d(S,T) = d(T,S)$
\item $d(S,W)  \le  d(S,T)  + d(T,W)$
\end{enumerate}
These properties are highly desirable for a shape comparison function. They imply that the
distance between shapes is stable and not overly sensitive to noise and measurement error.  If $S'$ is close to $S$ and
$T'$ is close to $T$, then condition (3) implies that $d(S,T) \approx d(S',T')$.
We will introduce a mathematical metric on the space 
$\mathcal S$ of genus-zero surfaces with piecewise-smooth Riemannian metrics, with two surfaces considered
equivalent if they are isometric.

Our method has many additional useful features.
It gives a conformal diffeomorphism from one surface to the other, useful for applications such as texturing. 
It is well suited to representation of smooth surfaces by triangular meshes. The computed correspondence is robust under a change of mesh, either from a perturbation of the location of vertices or
from using a combinatorially distinct mesh.
It  is intrinsic, so that the correspondence between two surfaces does not depend on how they are embedded in $\RR^3$, but only on their Riemannian metrics, and
 is therefore well suited for comparing flexible surfaces that arise in the study of  non-rigid objects.
The method  applies to immersed surfaces (surfaces with self-intersections) and to surfaces in arbitrary manifolds. Note that arbitrary Riemannian surfaces may
not be realizable as subsets of $\RR^3$. Furthermore, surfaces whose meshes have intersecting or overlapping triangles fit just
as well into the framework of our algorithm as embedded surfaces. 
Pairs of intersecting triangles are common in meshes constructed from point clouds, and
are problematic for some shape comparison approaches.
Finally, we note that our method can be Implemented to be completely automated, and does not rely
on any preliminary labeling of landmark or feature points. 
This allows for avoidance of errors and costs due to variability of human input.

Applications of our shape comparison method include:
\begin{enumerate} \itemsep 2pt
\item  Shape retrieval, or finding nearest fits in an atlas, or database of shapes,
\item  Geometric clustering,
\item  Alignment of surfaces with different conformations but similar surface geometry. For example, comparing scans of  faces that exhibit different facial expressions,
\item  Alignment of images of one object taken at different times, to measure change over time, and to locate subregions where changes have occurred,
\item Statistically sampling surfaces, and averaging  to find typical surfaces or random surfaces,
\item Determining the suitability of a conformal parametrization of a surface.  Computation of a very large dilation can indicate problems in conformal parametrization, resulting in a mesh that does not closely align to a modeled underlying surface.
\item Transferring a single common mesh to a collection of genus-zero surfaces. This in turn can be used to interpolate between collections of surfaces which
are initially described with distinct meshes, giving an average shape for a collection of differently meshed surfaces,
\item Creating a conformal map to use as an initial value or seed in shape correspondence methods that allow for non-conformal correspondences but depend
on a good initial correspondence,
\item Detecting symmetry.  If a surface has reflectional symmetry then it and a reflected copy have small distance. Similarly a  diffeomorphism whose source and image have small distance and that is not close to the identity indicates 
existence of a symmetry,
\item Coarsening a mesh while retaining surface geometry. The computation of $d_{sd}$ is minimally
 affected by subdivision or coarsening of a mesh, so coarsening a mesh will preserve $d_{sd}$ 
 as long as the coarsened mesh is geometrically close to the  original mesh.
\end{enumerate}
 
We consider here the case where 
each of the two compared shapes is a surface of genus zero,  or a topological sphere.  
The restriction to genus zero is appropriate for a wide variety of natural
surface comparison problems, including facial recognition,  alignment and comparison of brain cortices,  comparing protein surfaces, and geometric identification and comparison of
objects such as bones and teeth. Note that when a comparison is sought between two disk type surfaces, each with
a single boundary curve,
this problem can be transformed into a comparison of two spheres. 
The transformation can be accomplished, for example, 
by gluing a flat disk with appropriate boundary length  onto the boundary of each of the pair of
initial surfaces, turning them into genus-zero surfaces.
The same idea allows comparison of annuli or more general disks with holes.
Extensions to surfaces of higher genus can be carried out by
 considering conformal classes of flat and hyperbolic geometries, or by searching for
canonical surgeries to reduce a surface to genus zero. These will be explored elsewhere. 

Our approach is based on successively minimizing two energies defined on maps between surfaces.  We first minimize the Dirichlet energy among all maps between the surfaces, yielding a map which is harmonic.  
For genus zero surfaces, the harmonic maps exactly coincide with the
conformal maps, and this step reduces the maps to be considered from the
vast space of all diffeomorphisms to the much smaller, but still large, six-dimensional space of conformal maps. 
We then minimize again, this time using
a symmetrized energy function introduced in this paper that we call the {\em symmetric distortion energy}. 
Minimizing this energy amounts to picking an appropriate Mobius transformation, as indicated in Figure~\ref{fig:overview}.
The symmetric distortion energy gives a measure of the distance of a conformal map from an isometry.
This energy is both conceptually natural and efficiently computable, and achieves
good results in experimental tests. 
We show in Section~\ref{computations} that the symmetric distortion energy
and the metric on surfaces it induces behave well as measures of shape similarity. In a related paper we
apply the symmetric distortion distance to study similarities of shape in biological objects such as the surfaces of 
bones \cite{KoehlHassScience}.  Results in that paper  indicate that this distance is remarkably effective in distinguishing and grouping biological shapes.  

\begin{figure}[htbp] 
\centering
\includegraphics[width=4in]{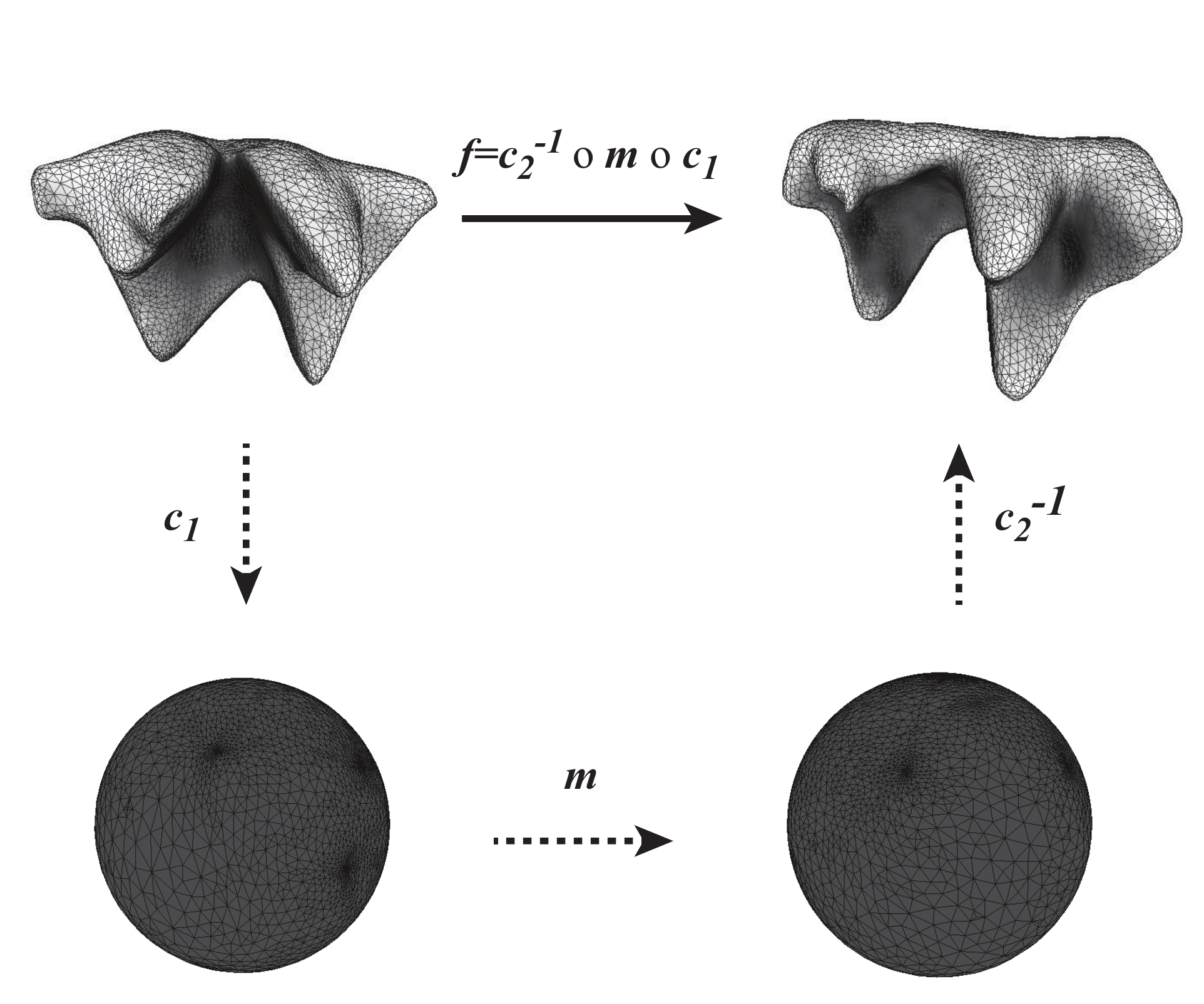} 
\caption{In this example the symmetric distortion distance between two teeth is calculated.
The first surface $ F_1 $ is scanned from the molar of a flying lemur while the second $ F_2 $ comes from a tree shrew.  
A conformal map minimizing  symmetric distortion energy is computed in two main stages. 
In the first stage, conformal maps to the round sphere,
$c_1: F_1 \to S^2$ and $c_2: F_2 \to S^2$, are computed.  In the second step, a Mobius transformation
 $m: S^2 \to S^2$ is computed that minimizes the symmetric distortion energy of the 
 composite map $ f = {c_2}^{-1 }  \circ m  \circ c_1 :  F_1 \to F_2$. The energy of $f$ gives the  symmetric distortion distance between the surfaces.}
\label{fig:overview}
\end{figure}

\subsection{Prior work}

Conformal maps from surfaces to the plane have become an important tool to visualize and to flatten surfaces,
in particular for surfaces that are topological disks, but also for spheres and higher genus surfaces. By mapping a 
surface region to the plane while preserving angles, these methods
allow for consistent visualization of locations on highly folded surfaces, and for graphical techniques
such as texturing. 

Pinkall and Polthier described a mid-edge method of computing discrete conformal maps and applied it to the computation of
discrete minimal surfaces \cite{PinkallPolthier}. Bobenko, Pinkall and Springborn gave an approach to computing discrete conformal maps based on an energy minimization technique \cite{BobenkoPinkallSpringborn}.
Thurston suggested that  discrete approximations of conformal maps could be obtained by circle packings.
This was carried out in work of Hurdal \cite{Hurdal} and Stephenson\cite{Stephenson}.  
Rodin and Sullivan, and He and Schramm established that the limits of discrete maps
obtained by circle packings converge to smooth conformal maps \cite{RodinSullivan:87}, \cite{HeSchramm:96}.
Haker et al. developed a method based on finite elements to compute discrete conformal maps  \cite{Haker}.
Gu and Yau computed discrete conformal parametrizations for surface matching \cite{GuYau:2002}. Jin, Wang, Yau and Gu used a stretching energy to create optimal parametrizations of surfaces \cite{Jin}. 
The Gromov-Hausdorff distance has also been used to develop shape comparison methods, as in \cite{BronsteinBronsteinKimmel, EladKimmel, Memoli}.

Recently a series of shape comparison methods introduced that are based on optimal
transport.
Lipman and Funkhouser developed a method to find an optimal conformal surface correspondence based on a voting scheme that weighs  transportation costs \cite{LipmanFunkhouser}.
Boyer et al. introduced several innovative methods to compare shapes 
based on minimizing a cost based on optimal transport \cite{Boyer}.   
They tested their methods on a collection of scanned biological objects, consisting of teeth, radius bones, and metatarsal bones from a variety of primates. 
This data was made available and we have used it to test our method and to compare its effectiveness to 
the methods described in  \cite{Boyer}, and to the expert observer data that they provided.

Earlier work of the two authors used related energies associated to conformal 
diffeomorphisms to compare shapes of brain cortices
and of protein surfaces \cite{KoehlHass:14}. This earlier work has been improved and further developed here.
In particular,  the optimal diffeomorphism produced by the symmetric distortion energy is proven to have values that 
give a metric on the space of shapes, a highly desirable feature not present in the previously studied energies. 
In the discrete setting, the approach given here has been improved to give
mesh independent surface comparisons. In contrast, 
the method given in \cite{KoehlHass:14} required combinatorially identical meshes 
before it could provide  a consistent measure of distances between shapes, 
a requirement that restricted the scope of applications.
These limitations have been overcome in the current work.

\section{Aligning smooth surfaces} \label{align}

In this section we develop our method  in the context of smooth surfaces and mappings.
This gives the underlying theory for the subsequent implementations of 
computational algorithms on triangulated or meshed surfaces.

A Riemannian surface  is a smooth  2-dimensional manifold equipped with a Riemannian metric,
a smoothly varying inner product on the tangent space of the surface.  
An {\em isometry} between two Riemannian surfaces is a map under which the Riemannian metrics correspond. 
In particular, an isometry preserves angles and  distances along the surface.
Not all angle preserving maps
are isometries.  Similarities of the plane, which stretch all distances uniformly, give an example of a non-isometric angle preserving map.
The maps that preserve angles at each point are called {\em conformal}.  

Some metric distortion is necessarily present in any construction of an alignment from a surface to another when no isometry exists.
 A measure of this distortion is given by the total stretching energy of the first surface as it is deformed
 over the second. This stretching can be measured by the {\em Dirichlet energy} $E_D(f)$ of a map $f : F_1 \to F_2$,  defined by the integral
\begin{eqnarray*}
E_D(f) =\frac{1}{2}  \int_{F_1}||df||^2 ~ dA   
\label{energy}
\end{eqnarray*}

The  maps that minimize Dirichlet energy between two surfaces are called {\em harmonic maps}. In two dimensions, 
the class of harmonic diffeomorphisms of spheres, the class of conformal diffeomorphisms, 
and  the class of holomorphic (complex differentiable) diffeomorphisms, all coincide.  
We  focus here on the angle-preserving property of conformal maps.  
 
A deep result, the {\em Uniformization Theorem}, states that a conformal diffeomorphism always exists between two smooth genus-zero Riemannian surfaces $F_1$ and $F_2$  \cite{Bers}.
However such conformal maps are not unique. 
Each conformal diffeomorphism $f : F_1 \to F_2$ is part of a 6-dimensional family. 
To understand this family we consider the case where $F_1$ is the round, radius-one 2-sphere $S^2 \subset \RR^3$. 
The space of conformal diffeomorphisms from $S^2$ to itself forms the six-dimensional group  $PSL(2,\CC)$, called the {\em Mobius} or {\em linear-fractional} transformations.
Any conformal map $f:S^2  \to F_2$ can be precomposed with a conformal Mobius transformation 
$\phi :S^2  \to S^2 $ to give a new conformal map $f \circ \phi : S^2  \to F_2$,
and this construction gives the entire six-dimensional family of  conformal maps from $S^2 $ to $F_2$
 
A conformal map $f:F_1 \to F_2$ stretches the metric of $F_1$  at a point $p \in F_1$ uniformly in all directions.
A conformal diffeomorphism then defines a real valued function $ {\lambda_f}: F_1 \to \RR^+$ that measures the
 stretching of vectors at each point.  
The function $ {\lambda_f}$ is called the  {\em dilation},
and is defined by
$$
f^*(g_2) =  {\lambda_f}^2 g_1,
$$
 where $g_1, g_2$ are the Riemannian metrics on $F_1, F_2$ respectively, and  $f^*(g_2)$ is the metric on
 $F_1$ obtained by pulling back the metric on $F_2$.
 This formula means that for $x \in F_1$, 
  vectors $v_1, v_2 \in T_xF_1$ in the tangent space of $F_1$ at $x$,
  and  $f_*(v_i)$ the image of $v_i$  in the tangent space of $F_2$ at $f(x)$ under the derivative of $f$, 
  we have 
 $$
 g_2 (f_*(v_1),f_*(v_2))_{f(x)} =   {\lambda_f}^2(x) g_1(v_1,v_2)_x .
$$
The conformal factor $ {\lambda_f}^2(x)$ measures the pointwise expansion or
 contraction  of area at $x$. 
 
A common measure of the global distortion of  a map $f : F_1 \to F_2$  is given by the {\em Dirichlet energy} of  $f $.
This energy is conformally invariant, meaning that is preserved by pre-composition with a conformal map.
The minimal value of the Dirichlet energy over all diffeomorphisms between a pair of smooth genus-zero surfaces is 
 obtained when $f$ is a conformal diffeomorphism, and in that case the energy is equal to the area of $F_2$.  
 For a conformal map with Jacobian determinant $|\mbox{Jac}(f)|$ we have  $||df||^2 = 2   {\lambda_f}^2 = 2 \|\mbox{Jac}(f)|$, 
 \begin{eqnarray*}
  E_D(f) =\frac{1}{2}  \int_{F_1} <df,df>  ~dA  ~~~ \ge~~~   \int_{F_1}|\mbox{Jac}(f)|~ dA = \mbox{Area}(F_2). 
\end{eqnarray*}
If we take $e_{1},e_{2}$ to be an orthonormal frame in a neighborhood of a point 
then the formula for the integrand on that neighborhood simplifies to
\[
\frac{1}{2} (||f_{\ast}(e_{1})||^{2}+||f_{\ast}(e_{2})||^{2})  \ge  ||f_{\ast}(e_{1})|| ||f_{\ast}(e_{2})|| \ge |\mbox{Jac}(f)|
\]
Equality holds  precisely when $f$  is a conformal diffeomorphism. 
Since the Dirichlet energies of any two conformal maps are equal, we introduce 
an additional, secondary energy to measure  the distance of a conformal map from an isometry.  This energy
emerges naturally from the following considerations.
If a  conformal map $f$ has constant dilation $ {\lambda_f} =1$, then $df$ preserves both lengths and angles at each point,
and thus $f$ is an isometry. So $| {\lambda_f} -1|$ indicates the pointwise deviation of a conformal map from an isometry. 
 This leads us to the following integral that globally measures this deviation,
 \begin{eqnarray*}
E_{el}(f) =   \int_{F_1} ( {\lambda_f}-1)^2 ~ dA.
\end{eqnarray*} 
We call this the {\em elastic energy} of the map $f$.
 
The elastic energy of a  conformal map is closely related to the average stretching of $\lambda_f$, given by
\begin{eqnarray*}
E_1(f) = \int_{F_1}  {\lambda_f}  ~ dA.
\end{eqnarray*}

\begin{lemma}  \label{E1}
A conformal map $f$ minimizes $E_{el}(f) $ among all conformal maps $f : F_1 \to F_2 $ 
if and only if $f$ maximizes  $E_1$ among all such  maps.
\end{lemma}
\begin{proof}
For a conformal map $f : F_1 \to F_2$, we have  $||df||^2 = <df,df> = 2  {\lambda_f}^2 = 2|\mbox{Jac}(f)|$
and $ E_D(f)  =  \mbox{Area}(F_2)$. Then
 \begin{eqnarray*}
E_{el}(f) &=&\int_{F_1} (\lambda_f-1)^2 ~ dA_1 \\
&=&  \int_{F_1}  ({\lambda_f}^2 -2\lambda_f +1) ~ dA_1 \\ 
&=&   \int_{F_1}  {\lambda_f}^2  ~ dA_1 + \int_{F_1}  1  ~ dA_1  -2 \int_{F_1}  \lambda_f  ~ dA_1  \\ 
&=&   E_D(f) +  \mbox{Area}(F_1)  -2  E_1(f) \\ 
&=&   \mbox{Area}(F_2) + \mbox{Area}(F_1)  -2  E_1(f).
\end{eqnarray*}
The first two terms do not depend on the choice of conformal map $f$.
Thus $E_{el}$ is minimized when $E_1$ is maximized.
\end{proof}

A somewhat similar functional was introduced by Jin, Wang, Yau and Gu \cite{Jin}. 
 The integrand in their work has the form $( {\lambda_f}^2-1)^2$ rather than the $( {\lambda_f}-1)^2$ used in
 our definition of elastic energy.
 They called this a ``uniformity energy'' on conformal maps and 
applied it to find optimal parametrizations of surfaces. 
However as a result of the fourth order term, Lemma~\ref{E1} does not apply to their energy. 

The energy $E_{el}$ is bounded above in the space of conformal maps from $F_1$ to   $F_2$  by $ A(F_2) +A(F_1) $.
The definition of elastic energy can be extended to non-conformal maps by taking
$$
E_{el}(f) =   E_D(f) +\mbox{Area}(F_1)  - \int_{F_1}\mbox{tr}(df )~ dA_1.
$$
We restrict attention to conformal maps in this paper.

A drawback of the elastic energy as a measure of shape distortion is its lack of symmetry. 
The elastic energy of $f$ and $f^{-1}$  may not be equal.  Nor does the 
optimality of $f$ imply the same for  $f^{-1}$.
We now define a symmetrized energy that corrects this shortcoming.\\ \\
{\bf Definition.} The {\em Symmetric Distortion Energy} of a conformal diffeomorphism $f : F_1 \to F_2$ with dilation function
$\lambda_f : F_1 \to \RR$ is

\begin{eqnarray*}
E_{sd}(f) & = &   \sqrt{E_{el} (f)} +    \sqrt{ E_{el} (f^{-1} ) }   \\
 & = &   \sqrt{  \int_{F_1} (\lambda_f (z) -1)^2 ~ dA_1 } +    \sqrt{   \int_{F_2} (  \lambda_{f^{-1}}(z)-1)^2  ~ dA_2 }   .   \label{eqn:metric.energy}
\end{eqnarray*}

Note that $E_{sd}(f) = E_{sd}(f^{-1}) $.
A conformal map that minimizes $E_{sd}(f) $ gives the optimal correspondence between $F_1$ and $F_2$
that we seek, and the magnitude of $E_{sd}$  for such an optimizing map 
defines the distance between the two surfaces. 

Our next goal is to show the existence of  a diffeomorphism that minimizes 
$E_{sd}(f)$. A difficulty is that the pointwise limit of a sequence 
of conformal diffeomorphisms may be discontinuous, or may map all of $F_1$ to a single point in $F_2$.
We need to show that a sequence approaching an infimum of the energy does not have this
undesired behavior.
We first examine the special case when $f$ is a conformal map from the round sphere to itself.

\begin{lemma}  \label{Mobius}
If a sequence of Mobius transformations $m_i : S^2 \to S^2$ has no convergent subsequence, 
then there is a subsequence for which  $\lim_{i \to \infty} E_1(m_i)\ =   0 $.
 \end{lemma}
\begin{proof}
A Mobius transformation is completely determined by the image of three points. Let
$P, Q, R$ be any there points in $S^2$.  If a sequence
of Mobius transformations $\{m_i\}$ takes $P, Q, R$  to points $P_i, Q_i, R_i$, and these converge to 
three distinct points $P', Q', R'$, then  the sequence
of Mobius transformations $\{m_i\}$ is equicontinuous and converges to the unique Mobius transformation
that takes $P, Q, R$ to $P',Q',R'$.  Thus if $\{m_i\}$  has no convergent subsequence then
two or more of $P, Q, R$ are converging to a single point as $i \to \infty$.

A nontrivial Mobius transformation fixes either one (called parabolic) or two (called elliptic or hyperbolic) points on $S^2$.
We divide the proof into three cases, according to the limiting behavior of these fixed points as $ {i \to \infty}$.\\
Case (1): There is a subsequence in which each Mobius transformation has a single fixed point. \\
Case (2):  There is a  subsequence  in which each  Mobius transformation
 has two distinct fixed points, and these converge to two distinct points as $i \to \infty$. \\
Case (3):  There is a subsequence  in which each Mobius transformation has two distinct fixed points, and these converge to a single point as $i \to \infty$. \\
One of these three cases must hold;  if Case (1) does not apply, then all but finitely many Mobius transformations have two distinct fixed points,
and by compactness  either  Case (2) or Case (3) hold for some subsequence.
In each case we show that the values of $E_1$ on the subsequence limit to 0.

The Riemannian metric of the round sphere, with the north pole removed, is isometric 
under stereographic projection to the plane $\RR^2 $ with the Riemannian metric
\begin{eqnarray*}
ds^2 = \frac{4}{(1 + x^2 + y^2)^2} (dx^2 + dy^2  )   .
\end{eqnarray*}
The area form for this metric is
\begin{eqnarray*}
dA_1 =   \frac{4}{(1 + r^2)^2} dA  =   \frac{4}{(1 + r^2)^2} (rdr  \wedge d\theta   ) .
\end{eqnarray*}

\noindent
Case (1): Take a subsequence of $m_i$ consisting of Mobius transformations with a single fixed point and rotate $S^2$ so that each $m_i$ fixes  the north pole.  Note that conjugating by a rotation does not change $E_1(m_i) $ and that we still have no converging subsequence since $S^2$ is compact.
Then each  $m_i$ fixes $\infty$ in the coordinates given by stereographic projection and thus acts as a translation of the plane. By further conjugating by a rotation fixing both 0 and $\infty$, we can arrange that
that $m_i(z) = z+B_i$, with $B_i>0 \in \RR$.  
Divergence of  $\{ m_i \}$ implies that  $B_i \to \infty$. A computation gives that
\begin{eqnarray*}
\lambda_{m_i}(z)=   \frac{1 + |z|^2} {1 + |z+B_i|^2} .
 \end{eqnarray*}

We first consider the integral  of $E_1(m_i)$  over the disk
$D_0 = D(0,B_i/2)$   of radius $B_i/2$ centered at  the origin.
On $D_0$ we have $|z+B_i| \ge B_i/2$, so that

\begin{eqnarray*}
 \int _{D_0} \lambda_{m_i} ~ dA_1   &=&    \int_{D_10}  \frac{1 + |z|^2} {1 + |z+B_i|^2 } \frac{4}{(1 + |z|^2)^2}  ~dA  \\
 &=&   \int_{D_0}    \frac{1 } {1 + |z+B_i|^2 } \frac{4}{1 + |z|^2}  ~dA  \\
 &\le& 4 \int_{D_0} \frac{1 } {1 + |B_i/2|^2 } ~ \frac{1}{1 + |z|^2}  ~dA \\
 &=&   \frac{4\pi} {1 + |B_i/2|^2 }   \int_0^{B_i/2}  \frac{2r}{(1 + r^2)}    ~ dr  \\
&=&   \frac{4\pi} {1 + |B_i/2|^2 }   \log (1 + {B_i}^2 /4).
\end{eqnarray*}
Note that this integral approaches 0 as $ {B_i}\to \infty$.

Next we compute $E_1(m_i)$ over the disk  $D_1  = D((-B_i,0),B_i/2)$  of radius $B_i/2$ centered at $(-B_i,0)$.
The change of variables $z = -w-B_i$ takes  $D_0$ to $D_1$ and we carry out the corresponding change of variables. 
 \begin{eqnarray*}
 \int_{D_1}    \frac{1 } {1 + |z+B_i|^2 } \frac{4}{1 + |z|^2}  ~dA  &=&  \int_{D_0}    \frac{1 } {1 + |w|^2 }~ \frac{4}{1 + |w+B_i|^2}  ~dA 
\end{eqnarray*}
As before, for $w \in D_0$   we have $|w+B_i| \ge B_i/2$, and
again this integral approaches 0 as $ {B_i}\to \infty$.

Finally we consider the integral of $E_1(m_i)$  over $D_2 = \CC - \{ D_0 \cup D_1 \}$. 

\begin{claim} $ |z+B_i| / |z| \ge 1/3$ for  $z \in D_2$. 
\end{claim}
\begin{proof} If  $|z| \le 3B_i/2$ then 
$$
z  \notin D_1 \implies z+B_i \notin D_0 \implies \frac {|z+B_i|}{ |z| } \ge \frac {|B_i/2|}{|3B_i/2|} = \frac {1}{3}.
$$
If $|z| \ge |3B_i/2|$ then $B_i \le 2|z|/3$ and 
$|z+B_i| \ge |z|/3$ so  again $ |z+B_i| / |z| \ge 1/3$.
\end{proof} 

Then
\begin{eqnarray*}
  \int_{D_2}    \frac{1 } {1 + |z+B_i|^2 } ~\frac{4}{1 + |z|^2}  ~dA &\le& 4 \int_{D_2}    \frac{1 } {1 + |z/3|^2 }~ \frac{1}{1 + |z|^2}  ~dA \\
&\le&4  \int_{\CC - D_0}    \frac{9 } { |z|^4 }    ~dA \\
&=&   72\pi   \int_{B_i/2}^\infty     \frac{1 } { r^4 }   ~ rdr  \\
&=&     \frac{144\pi  } {{B_i}^2 } .
\end{eqnarray*}
This term also approaches 0 as $ {B_i} \to \infty$.

Since $D_0 \cup D_1 \cup D_2 = \CC$ and the integral over each approaches 0 as $ {B_i}\to \infty$,
we have shown in Case (1) that $\lim_{i \to \infty} E_1(m_i) = 0$.

\noindent
Case (2):   
By conjugating with rotations as before, we can assume
that each $m_i$ has $\infty$ as an attracting fixed point and also fixes  a second point $p_i$.
In Case (2) we assume that $\lim_{i \to \infty} p_i \ne \infty$, and therefore after passing to a subsequence
we can assume $\lim_{i \to \infty} p_i = p$ for some point $p$.
We can conjugate  $m_i$ by the Mobius transformation $z \to z- p_i$ to get a Mobius transformation 
that fixes $\infty$ and the origin.
This gives a new sequence of Mobius transformations $m_i' $ that fixes both the origin and $\infty$    and
such that $ \alpha E_1(m_i)  <  E_1(m_i')  < \beta E_1(m_i) $ for fixed positive
constants $ \alpha, \beta$. Thus 
$$
\lim_{i \to \infty} E_1(m_i')\ =   0  \Leftrightarrow  \lim_{i \to \infty} E_1(m_i)\ =   0 
$$ 
and we can assume that $m_i$ fixes the origin and $\infty$.

A conformal transformation of the sphere fixing the origin and $\infty$ with $\infty$ as an attracting fixed point has
the form  $m_i(z) = A_i z,  |A_i| > 1 $.  
Since $\{m_i\}$ has no convergent subsequence we must have that $|A_i| \to \infty$. 
The value of $\lambda_{m_i}$ at $z$ is given by
\begin{eqnarray*}
\lambda_{m_i}(z)  = \frac{1 + |z|^2}{1 + |A_i|^2|z|^2} |A_i|.
\end{eqnarray*}
 
\begin{eqnarray*}
E_1(m_i) & = &  \int _0^{2\pi}  \int_0^\infty \lambda_{m_i} ~ \frac{4}{(1 + r^2)^2} ~ rdr d\theta   \\
& = &  2\pi   \int_0^\infty  \frac{(1 + r^2) |A_i|}{1 +  |A_i|^2r^2}  \frac{4}{(1 + r^2)^2}  ~ rdr  \\
& = &2\pi   \int_0^\infty    \frac{4r |A_i|} {(1 + |A_i|^2r^2)(1+r^2)}    ~ dr  \\
& = &\frac{8\pi  |A_i| \log( |A_i|)  } {|A_i|^2-1 } .
 \end{eqnarray*}

 Again $ \lim_{i \to \infty} E_1(m_i) = 0$.

\noindent
Case (3):  As before, we can first rotate $S^2$ so that each $m_i$ has the north pole as an attracting fixed point.
Then  $m_i (z) = A_i z + B_i$  for some $A_i,B_i \in \CC$ with $|A_i| \ge 1$.
By conjugating by a rotation fixing 0 and  $\infty$ we can assume that $B_i \in \RR^+$.
The second fixed point of  $m_i$  is then $P_i = -B_i/(A_i-1)$, and since we assume in Case (3) that $P_i \to \infty$, we have   $\lim_{i \to \infty} |B_i/(A_i-1)| = \infty .$  If $B_i$ is bounded then $A_i \to 1$ and the sequence $m_i = A_i z + B_i$ has a convergent subsequence,
contrary to our assumption. Therefore  $B_i \to \infty$ and hence also $B_i /|A_i| \to \infty$.

We first compute $E_1$  over the disk $D_1  = D(-B_i/A_i, B_i/|2A_i|)$.
The change of variables  $z=  -w/A_i - B_i/A_i$, with Jacobian $1/|A_i|^2$,
transforms this to an integral over  $D_0 =  D(0,  B_i/2 )$.  Note that for $w \in D_0$ we have 
$|w  +B_i| \ge  B_i/2.$
\begin{eqnarray*}
 \int _{D_1}\lambda_{m_i}(z) ~ dA_1  & = &  \int _{D_1}  \frac{|A_i| } {1 + |A_iz + B_i|^2 } ~ \frac{4}{1 + |z|^2}  ~dA \\
& = & \int _{D_0}  \frac{|A_i| } {1 + |w|^2 }  ~\frac{4}{1 + |w/A_i + B_i/A_i |^2} ~\frac{1}{|A_i|^2}  ~dA\\
& = & \int _{D_0}  \frac{|A_i| } {1 + |w|^2 }  ~\frac{4}{|A_i|^2 + | w + B_i |^2}  ~dA\\
& \le & \int _{D_0}  \frac{|A_i| } {1 + |w|^2 }  ~\frac{4}{|A_i|^2 + |B_i/2|^2}  ~dA\\
& = &   \frac{4 |A_i|}{|A_i|^2 + |B_i/2|^2}   \int _0^{2\pi} \int_0^{ \frac{B_i}{2}}  \frac{1 } {1 + r^2}  ~r dr d\theta \\
& = &   \frac{4 \pi |A_i|}{|A_i|^2 + |B_i/2|^2}   \int_0^{ \frac{B_i}{2}}  \frac{2r } {1 + r^2}  dr  \\
& = &    \frac{4 \pi |A_i|  \log ( 1 +   B_i^2 /4  )}{|A_i|^2 + |B_i/2|^2} \\
&\le &    \frac{16 \pi |A_i|  \log ( 1 +   B_i^2 /4  )}{  |B_i|^2} .
\end{eqnarray*}
Since  $|A_i| / B_i  \to 0$, this quantity approaches 0 as $  i \to \infty$.

 Over the disk $D_2 = D(0, B_i/|2A_i|)$ we have   $ |A_iz + B_i| \ge  B_i/2$.
 \begin{eqnarray*}
 \int _{D_2} \lambda(m_i ) ~ dA_1  & = &  \int _{D_2}  \frac{|A_i| } {1 + |A_iz + B_i|^2 } ~ \frac{4}{1 + |z|^2}  ~dA \\
& \le & \frac{|A_i| } {1 +  |B_i/2|^2 } \int _{D_2} ~ \frac{4}{1 + |z|^2}  ~dA \\
& = & \frac{2\pi |A_i| } {1 +  |B_i/2|^2 } \int_0^{B_i/|2A_i|}   ~ \frac{4}{1 + r^2 }  ~ rdr  \\
& = & \frac{4\pi |A_i| } {1 +  |B_i/2|^2 } \log ( 1 +   (B_i/|2A_i|)^2) \\
& \le & \frac{16 \pi |A_i| } { B_i}\frac{ \log ( 1 +    B_i^2) } { B_i}   .
\end{eqnarray*}
This also approaches 0 as $  i \to \infty$.

Finally we consider the region $D_3 = \CC - \{ D_1 \cup D_2\}$.

\begin{claim} In $D_3$ we have $|A_iz + B_i| \ge  B_i/2 $ and $|A_iz + B_i| \ge   |z /3| $.
\end{claim}
\begin{proof}
The first inequality follows from $z \notin D_1 \implies A_iz + B_i \notin D_2 \implies |A_iz + B_i| \ge  B_i/2 $.

We consider two cases for the second inequality.
If  $|z| \le 3B_i/|2A_i|$ then  $ z \notin D_1 \implies A_iz+B_i \notin D_0$ and
$$
\frac{ |A_iz+B_i| }{ |z| } \ge  \frac{ B_i/2 }{3B_i/|2A_i|} = \frac{  |A_i|}{3} \ge \frac{ 1}{3}.
 $$
If $|z| \ge  3B_i/|2A_i|$ then 
$$
 B_i  \le \frac{2 |A_iz |}{3} \implies  |A_iz +B_i|    \ge \frac{|A_iz |}{3} \ge \frac{|z|}{3},
$$
giving the second inequality.
\end{proof} 
 
Then
\begin{eqnarray*}
 \int _{D_3} \lambda_{m_i}(z)  ~ dA_1  & = &  \int _{D_3}  \frac{|A_i| } {1 + |A_iz + B_i|^2 } ~ \frac{4}{1 + |z|^2}  ~dA \\
& \le & \ \int _{D_3}  \frac{|A_i| } { |A_iz + B_i|^2 } ~ \frac{4}{ |z|^2}  ~dA \\
& \le & \int _{D_3}  \frac{|A_i|} {  |z /3|(B_i/2 ) } ~ \frac{4}{ |z|^2}  ~dA \\
& \le & 24\int _{\CC -D_2}  \frac{|A_i|} {  B_i   |z|^3 } ~  dA \\
 & = & \frac{ 48\pi |A_i|}{B_i}   \int _{B_i/|2A_i|}^\infty   \frac{1} { r^3 }  ~r dr   \\
& =& \frac{ 96 |A_i|^2  }{ {B_i}^2}.
\end{eqnarray*}
Since $|A_i|/B_i \to 0$, this also approaches 0 as $i \to \infty$. 
 
We have shown in all cases that $\lim_{i \to \infty} E_1(m_i)\ =   0 $, proving the Lemma.
\end{proof}

We now show that 
$E_{sd}$ achieves a minimum for an appropriate choice of conformal map $ f_0: F_1 \to F_2$.  
Let 
\[
{\mathcal I} = \inf \{ E_{sd} (f): f : F_1 \to F_2 \mbox{ is a conformal diffeomorphism}\} 
\]
A {\em minimizing sequence } of conformal maps $  \{ f_i : F_1 \to F_2 \} $  is a sequence with non-increasing values
for $E_{sd}$ such that
\[
\lim_{i \to \infty} E_{sd} (f_i) = {\mathcal I}.
\]

\begin{theorem}
There exists a conformal diffeomorphism  $f : F_1 \to F_2$ with  $E_{sd} (f_i) = {\mathcal I}$.
\end{theorem}
\begin{proof}
By the Uniformization Theorem we know there exist conformal diffeomorphisms $c_1 : F_1 \to S^2$ and $c_2 : F_2 \to S^2$. 
The set of all  conformal diffeomorphisms
from $F_1$ to $F_2$ is given by  maps 
$$ 
f = c_2^{-1}  \circ m   \circ c_1 : F_1 \to F_2,
$$ 
where $m :  S^2 \to S^2$ is a Mobius transformation. Thus we need to show that an appropriate choice of Mobius transformation $m_i$ gives a minimizer for $E_{sd}$.

Let $\{ f_i : F_1 \to F_2 \}$ 
be a  minimizing sequence. Then for each $f_i$ we have  $E_{sd} (f_i)  \le E_{sd} (f_1)$
We need to show that $\{ f_i \}$ has a  subsequence converging to a map $f_0$ with $E_{sd} (f_0) = I$.
For each $i$ we can write $ f_i = c_2^{-1}  \circ m_i   \circ c_1 $ for some
Mobius transformation $m_i :  S^2 \to S^2$.
For the maps $ f_i  = c_2^{-1} m_i c_1 : F_1 \to F_2, $ we have 
$$
E_1(f_i) = \int_{F_1} \lambda_{f_i}  ~ dA = \int_{F_1} \lambda_{c_1}  \lambda_{m_i}  \lambda_{c_2^{-1}}  ~ dA 
$$
By compactness of $F_1, F_2$ and $S^2$, there are positive  constants $a,A$ and $b,B$ such that $a < \lambda_{c_1}  < A$ and  $b <   \lambda_{c_2^{-1}} < B$. Letting $c=ab $ and $C = AB$ it follows that
$$
c E_1(f_i) < E_1(m_i)  < C E_1(f_i).
$$
In particular, 
\[ E_1(f_i)\to 0  \Leftrightarrow   E_1(m_i)\to 0    \Leftrightarrow  E_1(m_i^{-1})\to 0   \Leftrightarrow    E_1(f_i^{-1})\to 0  .
\]
Assume now that $f_i$ has no convergent subsequence.
Then neither does $m_i$ or $f_i^{-1}$, and Lemma~\ref{Mobius} implies that both $E_1(f_i) \to 0$ and $E_1(f_i^{-1}) \to 0$. 
Recall that 
\begin{eqnarray*}
E_{sd}(f) & = &   \sqrt{E_{el} (f)} +    \sqrt{ E_{el} (f^{-1} ) }   \\
& = &   \sqrt{\mbox{A}(F_2) + \mbox{A}(F_1)  -2  E_1(f)} +    \sqrt{  \mbox{A}(F_2) + \mbox{A}(F_1)  -2  E_1 (f^{-1} ) }  
\end{eqnarray*}
so that 
\[
E_{sd} (f_i)  \to  2  \sqrt{ \mbox{Area}(F_2) + \mbox{Area}(F_1}) .
 \]
Now $E_{sd} (f_1)  <  2  \sqrt{ \mbox{Area}(F_2) + \mbox{Area}(F_1}) - \epsilon$ for some $\epsilon > 0$ so that
for $i$ sufficiently large,
 $E_{sd}  (f_i)  > E_{sd} (f_1) $.  This cannot happen for a minimizing sequence, and thus we have a contradiction
 to  the assumption that $\{f_i\}$ does not have a convergent subsequence. 

A convergent subsequence of a minimizing sequence gives a new minimizing
sequence. A convergent sequence of Mobius transformations that limit to a Mobius transformation converges
smoothly to the limiting map, and as a consequence the conformal maps $\{ f_i \}$ also
converge smoothly to a limiting map $f_0$, whose energy is equal to ${\mathcal I} $.
 \end{proof}

\subsection{A metric on shape space}

We now show  that $E_{sd}$ gives a metric on the space of spherical shapes $\mathcal S$.
 
Define a distance function $d_{sd}: \mathcal{S}\times \mathcal{S} \to \mathcal{R}$ by taking the minimal symmetric energy over the space 
 of conformal maps:
\begin{eqnarray*}
d_{sd}(F_1, F_2) =  \mathcal{I} = \inf  \{ E_{sd}(f) | f : F_1 \to F_2 \quad  \text{is a conformal diffeomorphism}  \} .
\end{eqnarray*}
 \begin{theorem}
The function $d_{sd}$ defines a metric on $\mathcal S$.
\end{theorem}
\begin{proof}
Let $F_1, F_2, F_3$ be  three genus-zero smooth surfaces.
To  show that $d_{sd}$ is a metric we need to check that:
\begin{enumerate}
\item $d_{sd}(F_1,F_2) \ge 0$
\item  $d_{sd}(F_1,F_2)= 0 \iff $ $F_1$ and $F_2 $ are isometric.
\item   $d_{sd}(F_1 ,F_2)=  d_{sd}(F_2, F_1)$
\item $d_{sd}(F_1 ,F_3) \le  d_{sd}(F_1 ,F_2) + d_{sd}(F_2, F_3)$
\end{enumerate}
The first three properties are direct consequences of Theorem 2.5 and  of the formula for the symmetric distortion energy $E_{sd}$.
We now establish  the triangle inequality.
Suppose that  $f : F_1 \to F_2$ and $g : F_2 \to F_3$ are conformal diffeomorphisms with dilations 
$  \lambda_{f} $ and $ \lambda_{g} $. Note first that $g\circ f:F_1 \to F_3$ is a conformal diffeomorphism with dilation $\lambda_{f}\lambda{g}$. We establish first the following lemma.
\begin{lemma}
\label{lemm:l1}
\begin{eqnarray*}
\sqrt{E_{el}(f)} + \sqrt{E_{el}(g)} \ge \sqrt{E_{el}(g\circ f)}
\end{eqnarray*}
\end{lemma}
\begin{proof}
Let $K = \sqrt{E_{el}(f)} + \sqrt{E_{el}(g)}$. Then
\begin{eqnarray*}
K^2 &=& E_{el}(f) + E_{el}(g) + 2 \sqrt{E_{el}(f)} \sqrt{E_{el}(g)} \\
&=& \int_{F_1} (1- \lambda_f)^2~ dA_1    + \int_{F_2}(1-\lambda_g)^2   ~ dA_2 
 + 2   \sqrt{  \int_{F_1} (1- \lambda_f)^2 ~ dA_1       \int_{F_2} (1-\lambda_g)^2   ~ dA_2    }.
 \end{eqnarray*}
 A change of variables in the second term gives
\begin{eqnarray*}
 \int_{F_2} (1-\lambda_g )^2 ~ dA_2  =   \int_{F_1} (1-\lambda_g)^2   ~\mbox{Jac}(f)~ dA_1   = 
  \int_{F_1} (1 - \lambda_{g} )^2    \lambda_f^2    ~ dA_1.
\end{eqnarray*}
Therefore,
\begin{eqnarray*}
K^2=\int_{F_1} (1- \lambda_f)^2    + (1-\lambda_g)^2    \lambda_f^2    ~ dA_1 
 + 2   \sqrt{  \int_{F_1} (1- \lambda_f)^2 ~ dA_1       \int_{F_1} (1-\lambda_g)^2    \lambda_f^2    ~ dA_1    }.
\end{eqnarray*}
The Cauchy-Schwartz inequality can be applied to the second term on the right side of the equation above,
\begin{eqnarray*}
   \sqrt{  \int_{F_1} (1- \lambda_f)^2 ~ dA_1   \int_{F_1} (1-\lambda_g)^2    \lambda_f^2    ~ dA_1    }.
\ge  \int_{F_1}    (1- \lambda_f) (1-\lambda_g)   \lambda_f  ~ dA_1 .
\end{eqnarray*}
Therefore,
\begin{eqnarray*}
K^2 &\ge& \int_{F_1} (1- \lambda_f)^2    + (1-\lambda_g)^2    \lambda_f^2    ~ dA_1  + 2  \int_{F_1}    (1- \lambda_f) (1-\lambda_g)   \lambda_f  ~ dA_1 \\
&\ge& \int_{F_1} (1- \lambda_f)^2    + (1-\lambda_g)^2    \lambda_f^2 +2(1- \lambda_f) (1-\lambda_g)   \lambda_f   ~ dA_1.
\end{eqnarray*}
Expansion of this last integrand gives:
\begin{eqnarray*}
&&(1- \lambda_f)^2    + (1-\lambda_g)^2    \lambda_f^2  +2(1- \lambda_f) (1-\lambda_g)   \lambda_f  \\
&=& 1-2\lambda_f + \lambda_f^2 + \lambda_f^2 -2\lambda_g \lambda_f^2 + \lambda_g^2\lambda_f^2 +2\lambda_f -2\lambda_f^2 -2\lambda_f\lambda_g +2\lambda_f^2\lambda_g \\
&=& 1 - 2\lambda_f \lambda_g + \lambda_f ^2\lambda_g ^2\\
&=& (1-\lambda_f\lambda_g)^2.
\end{eqnarray*}
Therefore,
\begin{eqnarray*}
K^2 &\ge& \int_{F_1}  (1-\lambda_f\lambda_g)^2 ~dA_1\\
&\ge& E_{el}(g\circ f) .
\end{eqnarray*}
As $K$ and $E_{el}(g\circ f)$ are positive, this concludes the proof.
\end{proof}
We  now return to the proof of Theorem 2.6. 
From Lemma \ref{lemm:l1}, we have:
\begin{eqnarray*}
\sqrt{E_{el}(f)} + \sqrt{E_{el}(g)} \ge \sqrt{E_{el}(g\circ f)} .
\end{eqnarray*}
Similarly,
\begin{eqnarray*}
\sqrt{E_{el}(f^{-1})} + \sqrt{E_{el}(g^{-1})} \ge \sqrt{E_{el}(f^{-1}\circ g^{-1})} .
\end{eqnarray*}
By summing these two inequalities, we get,
\begin{eqnarray}
E_{sd}(f) + E_{sd}(g) \ge E_{sd}(g\circ f) .
\label{eqn:esd}
\end{eqnarray}
Now let us suppose that $f_1$ is a diffeomorphism that realizes $d_{sd}(F_1,F_2)$ and $g_1$ 
a diffeomorphism that realizes $d_{sd}(F_2,F_3)$, so that
 $d_{sd}(F_1,F_2)=E_{sd}(f_1)$ and $d_{sd}(F_2,F_3)=E_{sd}(g_1)$. 
From Equation (\ref{eqn:esd}) we conclude that
\begin{eqnarray*}
E_{sd}(g_1\circ f_1) \le d_{sd}(F_1,F_2) + d_{sd}(F_2,F_3) .
\end{eqnarray*}
Since $d_{sd}(F_1,F_3)\le E_{sd}(g_1\circ f_1)$, the triangle inequality holds for $d_{sd}$.
\end{proof}

\subsection{A scale invariant distance}

In many applications it is desirable to compare two shapes that are defined only 
up to scale.  Sometimes data is presented
in units which are unknown and cannot be compared to standardized units.
It can also be useful to separate scale from other aspects of shape comparison.
For this reason we introduce a scale invariant version of the symmetric distortion metric, which we call
the  {\em normalized symmetric distortion metric}.
The computation of the normalized symmetric distance of two surfaces is found by first rescaling
so that each has surface area equal to one, and then computing the symmetric distortion metric as before.
This normalized symmetric distortion metric will be
used in the mathematical and biological measurements that we describe in Section~\ref{computations}.

\subsection{Other energies}
Arguments  similar to those of Lemma~\ref{Mobius} imply that many other energy functions on the space of conformal maps also realize a minimum value on some explicit conformal map. It is not immediately clear whether these energies define a distance metric on 
the space of genus-zero surfaces.

{\bf Definition.} The {\em $L^p$ distortion energy} of a conformal diffeomorphism $f : F_1 \to F_2$ with dilation function $\lambda = e^u$ and $ 1 \le p < \infty$ is given by
 \begin{eqnarray*}
E_{p}(f) =  \left[ \int_{F_1 } |u(x)|^p ~ dA_1  \right]^{1/p} .   
\end{eqnarray*} 
An energy-minimizing conformal diffeomorphism exists for each $E_{p}$, since
if a minimizing sequence has no convergent subsequence then $ \lambda \to 0$ and
$|u | \to \infty$ away from small neighborhoods of two points.  It then follows
that $E_{p}(f)  \to \infty$ for each $p$. The contribution of the inverse can be added in as before to
give a symmetrized energy, which is also unbounded outside of a compact set of Mobius transformations,
 \begin{eqnarray*}
E_{sd_p}(f) =  \left[ \int_{F_1 } |u(x)|^p ~ dA_1  \right]^{1/p} +  \left[ \int_{F_2 } |u(y)|^p ~ dA_2  \right]^{1/p} .   
\end{eqnarray*}
Note that $f^{-1}$ has dilation $1/\lambda = e^{-u}$ so that the formula for $E_{p}(f^{-1}) $ 
appears similar to that of $E_{p}(f) $, but the function $|-u|$ is evaluated on $F_2$ for the case of $f^{-1}$.

\section{Meshed Surfaces}

In applications we generally work with surfaces described by meshes, or
piecewise-flat  triangulations, rather then smooth surfaces.
These metrized triangulations either have coordinates in $\RR^3$ given for each vertex, or have  a length given for
each edge. 
In either case  a metric is determined in which each triangle is flat and each edge has an assigned length.
The metric is smooth except at the vertices, where the surrounding angle may be less than $2\pi$.
A conformal map is approximated in this setting by an appropriately defined {\em discrete conformal map},
as described by Luo \cite{Luo}.

 Given a surface of genus zero $F_1$ with a metrized triangulation we compute a discrete conformal map 
 $f : F_1 \to S^2$ from $F_1$ to the unit 2-sphere in $\RR^3$  using the algorithm of  Bobenko, Pinkall and Springborn \cite{BobenkoPinkallSpringborn}. 
While we have adopted this procedure for the computations presented here, we note that other methods
of computing discrete conformal maps, such as circle packings or the discrete Ricci flow, can also be used.

We consider a triangular mesh $\mathcal{M}=(V,E,T)$  in a surface. We do not restrict its 
combinatorial type. 
The  geometry of the surface represented by $\mathcal{M}$ is encoded in its edge lengths. 
A \emph{discrete metric} on $M$ is a function $l$ defined on the set of edges $E$ of the mesh, which assigns to each edge $e_{ij}$ a length $l_{ij}$ so that the
triangle inequalities are satisfied for all triangles in $T$. 

When working with meshes and discrete metrics, the  elastic energy integral is approximated by a sum over the mesh.
We consider two triangular meshes $\mathcal{M}_1$ and $\mathcal{M}_2$  in ${\mathbb R}^3$  with possibly different combinatorics and different geometries. The geometries are encoded either in 
the positions of the vertices or an assignment of lengths to the edges.  
Given a transformation $f : F_1 \to F_2$, in \cite{KoehlHass:14} we worked with an  elastic energy given by
\begin{eqnarray*}
\sum_{e_{ij} \in E} \left(  \frac{ l(f(e_{ij})) } {l(e_{ij})}  -1 \right)^2,
\end{eqnarray*}
where $l(e_{ij})$  indicates the length of the edge $e_{ij}$ in $F_1$,   $ l(f(e_{ij})) $ the length of the image of this edge in $F_2$, 
 and the sum is over all edges of the mesh on $F_1$.  When $f$ is conformal and the mesh is close to uniform then the quantity $ l(f(e_{ij}))  / l(e_{ij}) $  approximates $\lambda_f$, and this sum is an approximation of the symmetric distortion energy. However the sum is dependent on the size of the mesh, increasing with the number of edges. 
To make this quantity mesh independent,
we  weight the terms bythe area of the region to which each edge contributes.
This leads to the following formula for a mesh independent elastic energy,
\begin{eqnarray*}
L(f) = \sum_{e_{ij} \in E} \left(  \frac{ l(f(e_{ij})) } {l(e_{ij})}  -1 \right)^2  \frac{A_{ij} }{3}.
\end{eqnarray*}
Here $A_{ij}$ is the sum of the areas of the two triangles adjacent to edge $e_{ij}$ and the sum is over all edges
of the mesh on $F_1$.
The weighting factor (1/3) assigns to each edge the portion of the area of the two adjacent triangles obtained by dividing the triangles into three pieces.
 
The symmetric distortion energy in the discrete setting is then obtained by summing over the edges $E$ of the $F_1$ mesh and the edges $E'$ of the mesh on $F_2$:

 \begin{eqnarray}
E_{sd} (f)   = \sqrt{ \sum_{e_{ij} \in E} \left(  \frac{ l(f(e_{ij})) } {l(e_{ij})} -1 \right)^2  \frac{A_{ij} }{3}} \\
+  \sqrt{ \sum_{{e'_{kn}} \in E'} \left( \frac{ l(f^{-1} (e'_{kn})) } { l(e'_{kn})} -1 \right) ^2 \frac{A_{kn} }{3}}
\label{eqn:energy}
\end{eqnarray}

\subsection{Procedure and Implementation}

We begin with two combinatorial surfaces $F_1, F_2$ with metrized triangulations $\tau_1, \tau_2$.  We then implement the following steps. The process is indicated in Figure~\ref{fig:overview}.
 
\begin{enumerate}

\item{Construct conformal maps to the unit sphere.} \\
We use the methods of \cite{BobenkoPinkallSpringborn}  to construct  discrete conformal  maps $c_1: F_1 \to S^2$ and $c_2: F_2 \to S^2$  from each of a  pair of  genus-zero surfaces $F_1, F_2$ to the unit sphere $S^2$. 

\item {Move the  centers of mass of the vertices to the origin.}\\
This step is done for numerical stability. We compose $c_1$ with a Mobius transformation $m_1$ and $c_2$ with a Mobius transformation $m_2$ so that the vertices of $m_1 \circ  c_1(\tau_1)$ and $ m_2 \circ c_2(\tau_2)$  have centers of mass at the origin.
This step is done to prevent a choice of a conformal map $c_1$ which pushes most of the vertices into a small neighborhood of one point on the sphere. Any choice of conformal map from $F_1 \to S^2$ and from $F_2 \to S^2$ is theoretically valid for our method, but some are computationally problematic.

 \item{Map the source mesh onto the target surface.}\\
A Mobius transformation $m: S^2 \to S^2$ induces a map $ {c_2}^{-1}  \circ m  \circ c_1$ of the vertices of $\tau_1$ to $F_2$.
Given a vertex $v_i$ in $F_1$, we identify its image $v'_i$ in the spherical mesh $ c_1(F_1) $. We then
 locate its image $v''_i=m (v'_i)$ on the spherical mesh ${c_2(F_2)}$ and transfer this point to the surface $F_2$ by applying $c_2^{-1}$.  
The image of a point that is not a vertex is specified using barycentric coordinates of the simplex that contains the point.
 
\item{Find an optimal M\"{o}bius transformation.}\\
We search for the M\"{o}bius transformation $m: S^2 \to S^2$ that gives rise to a closest to isometric mapping among conformal maps between the two surfaces of interest, by searching for a global minimum of $E_{sd} $ as given in Equation~(\ref{eqn:energy}). We obtain a candidate as the solution of a non-linear optimization problem, via a steepest descent approach to solve this problem. Steepest descent methods are generally fast, but sensitive to local minima  and thus dependent on  the choice of an initial approximation to a solution. 

A random or fixed initial guess, such as the identity transformation, is likely to lead to a non-optimal local minimum.
Each initial guess is determined by specifying the images of three fixed points on $F_1$.  We can get a 
collection of initial assignments comparable in density to the size $n$ of the mesh on $F_2$ by choosing all
possible assignments for these three points that send them to vertices of $F_2$. 
The set of possible choices is then $O(n^3)$, which is prohibitive for large meshes.
We use a  procedure developed in \cite{KoehlHass:14} to automatically generate a collection of reasonable initial starting points. 
The method uses  ellipsoid approximations to $F_1$ and $F_2$ to give the initial alignment. Each ellipsoid approximation generates six points on each surface, corresponding to extremal  points where the three coordinate-axis meet the surface.
We label these points $x_\pm^1, y_\pm^1, z_\pm^1$ and $x_\pm^2, y_\pm^2, z_\pm^2$ . An initial choice of Mobius transformation is uniquely determined by the image of three of the points on $F_1$.  We  have six choices of where to initially map $x_+^1$, namely any of $x_\pm^2, y_\pm^2, z_\pm^2$. The point $x_-^1$ is then assigned to the antipodal point on the $F_2$ ellipsoid. We then have four choices of where to map $y_+^1$, namely to any of the four points orthogonal to the image of the first point on the ellipsoid.  The image of $z_+^1$ is then determined by orientation. 
Thus we have a total of 24 choices of initial mappings that are orientation preserving.  Once the image of three points is
specified, a unique Mobius transformation is determined, and this is used as one of our collection of initial maps.

If we also want to consider orientation reversing correspondences, then we first reverse the sign of each $z$-coordinate of $F_1$ and then reapply the process using the reflected surface. This gives a total of up to 48 initial correspondences in the
unoriented case.

We then apply steepest descent based on the symmetric distortion energy  to find an optimal conformal transformation.
We use Equation~(\ref{eqn:energy}) to compute the symmetric distortion energy of $f$ and  $f^{-1}$, and the gradient of this energy to find a minimum value.  We compute the symmetric distortion distance as the smallest value found for the symmetric distortion energy, and use the associated diffeomorphism as an approximation of the  symmetric distortion energy minimizing map.
\end{enumerate}

\section{Geometric Computations} \label{computations}
In this section we explore the geometric meaning of the $d_{sd}$-distance 
by computing it for pairs of well-understood geometric objects.
This allows us to develop a sense of what $d_{sd}$ is measuring.
We measure the $d_{sd}$-distance between spheres of varyiing radii, ellipsoids of varying principal axes,
and surfaces of varying roughness.  We also study the effects of decreasing the density of a mesh and
of changing the orientation of a surface.
 
Features that we would like to see hold for $d_{sd}$ to allow its use as a robust shape measurement tool are:
\begin{enumerate}
\item High sensitivity to small changes in area,
\item High sensitivity to small changes in shape,
\item Low sensitivity to small amounts of noise,
\item Mesh independence,
\item Low sensitivity to deformations that preserve intrinsic surface geometry,
\item Ability to distinguish an object from its reflection.  
\end{enumerate}
We show by a series of computational experiments 
that $d_{sd}$ exhibits highly favorable behavior for each of these features.

Feature (1), sensitivity to area change,
can be valuable in some settings, such as measuring the growth of an organism or of a tumor over time.
In other settings we want to consider only shapes up to scale, such as when scans are obtained without
a consistent measurement scale.
If we want to ignore the effect of changing area, we can normalize all areas to one by rescaling. 
Feature (2), sensitivity to small changes in shape, can be measured in a variety of ways.  Below we 
investigate the effect on  $d_{sd}$-distance of the deformation of a sphere to an ellipsoid which is stretched along
one axis while maintaining constant area.
Feature (3) is essential for robust distance measurements that are not unduly affected
by small amounts to noise or measurement error.
Nose sensitivity is measured by looking at the effect on $d_{sd}$-distance of  random perturbations of the vertices
of a sphere. 
The mesh independence property of Feature (4) implies that $d_{sd}$-distances
 are not dependent on the choice of a mesh or triangulation used to represent a surface. 
 This allows for comparing the geometric similarity of objects having meshes of varying density and
combinatorial type, subject only to the mesh accurately representing the surface.
Feature (5) is important for the comparing of flexible surfaces, such as the surfaces of
proteins, faces and animals that take on different configurations or poses.
Feature (6) allows for the $d_{sd}$-distance to distinguish objects that differ only in 
chirality, such as left and right hands, or left and right molars.

\subsection{Area rescaling}
 
 In many applications shapes are presented without scales.  For example, two medical
 images produced with different machines can describe the same shape in different
 coordinates whose relative magnitudes is not known.  Thus it is often convenient to first
 rescale each of the two surfaces being compared so that they have the same area, which
 we can take to be  equal to one.
 
 However in some cases it is useful to measure the effect of a change of scale. For example one may want
 to measure the growth of an object over time. When scale is the only difference between
 two shapes $F_1$ and $F_2$ then $E_{sd}$  measures an integral of the stretching required to enlarge one to fit the other.
 The formula for the energy required to perform such ra rescaling can be directly computed.
The distance  $d_{sd}(S_1, S_2)$
 between spheres $S_1$ of area $A_1$ and $S_2$ of area $A_2$ whose optimal alignment is realized by rescaling can
 be computed using Equation~(\ref{eqn:metric.energy}), giving  
 $$
d_{sd}(S_1, S_2) = 2\left| \sqrt { A_2 } - \sqrt {A_1} \right|.
 $$
 
\subsection{Area preserving shape deformation}

To measure the effect of global changes in shape on the distance $d_{sd}$ between two surfaces of equal area, we 
ran a computation that measured the distance between surfaces in a family of ellipsoids
from the unit sphere in $\RR^3$.  Two of the principal axes of each ellipse are held fixed at radius one, 
while the third is varied from 1/100 to 10.
The areas of all surfaces are then normalized to one by appropriately rescaling the surfaces,
and the minimal symmetric distortion energy is then computed.
The results are indicated in Figure~\ref{fig:ellipse}, where the distance of each ellipsoid from the unit sphere is given as a function of the  length of the third axis.  Note that the distance increases linearly near the point where both surfaces are unit spheres, indicating that $d_{sd}$  has the ability to differentiate small changes in shape when the two surfaces are close to isometric.  This feature is highly desirable for the use of $d_{sd}$ as a  tool for classifying surfaces, as it shows that near-similar surfaces can be differentiated.  In contrast, the sphericity, a common measure of similarity to a round sphere that compares the isoperimetric ratio of a surface to that of a sphere,
 is  insensitive to small changes in shape near an isometry, as shown in  Figure~\ref{fig:ellipse}.
 
A limitation of the current implementation of our computation of  the $d_{sd}$-distance is visible in this experiment.  
Our method of discrete approximation involves measuring the effects of stretching edges of a mesh, 
and this leads to maps that try to avoid sending edges of the mesh far out into spikes or protrusions.
These issues occur in ellipsoids with one principal axis stretched by a factor close to 10, as shown in the graph of
 Figure~\ref{fig:ellipse} (C).

\begin{figure}[htbp] 
\centering
\includegraphics[width=5in]{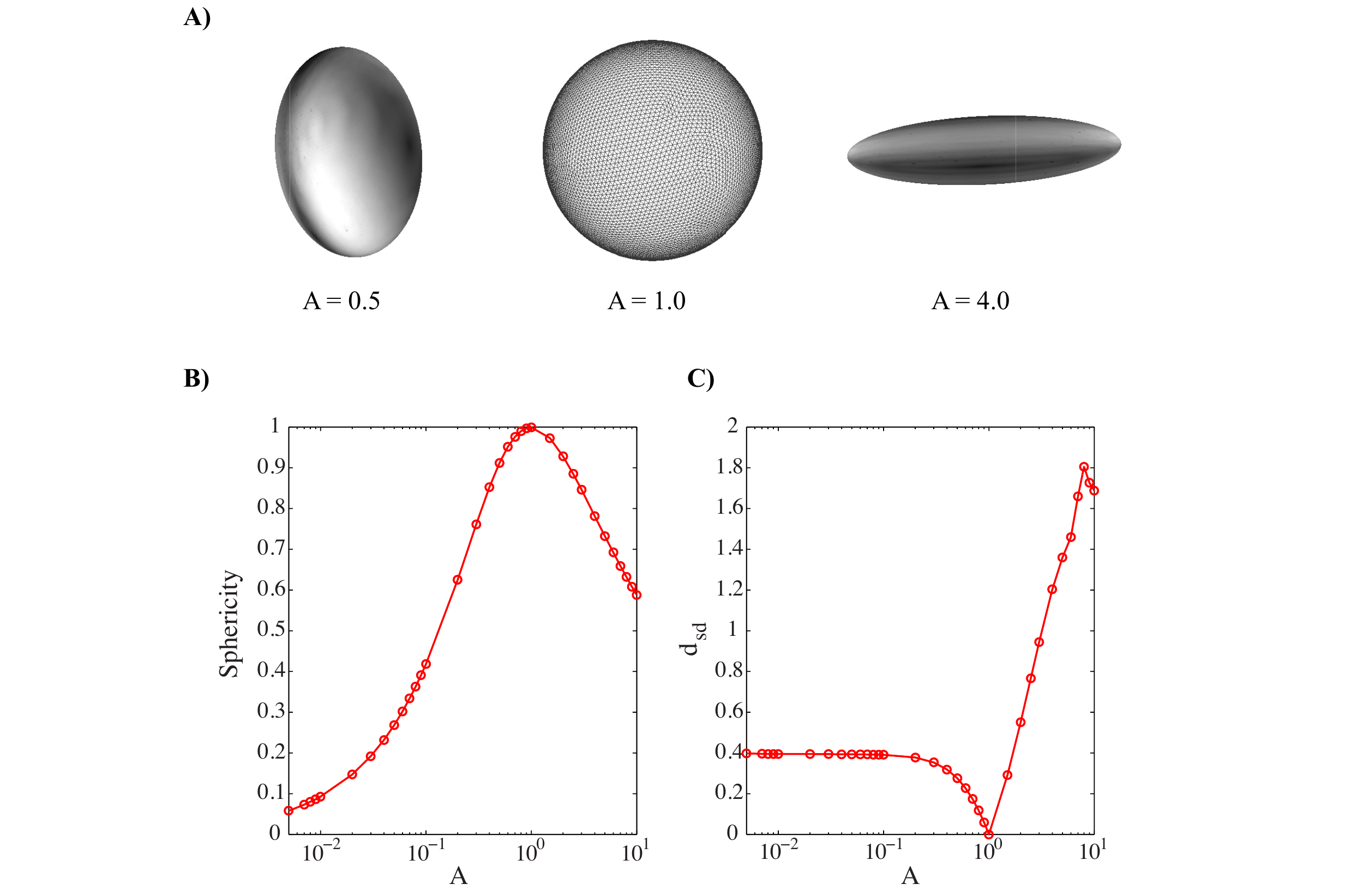} 
\caption{The effect of global shape on the distance between surfaces is indicated by deforming one axis of an ellipsoid and computing the distance to a unit sphere. All surfaces are  scaled to have area one, so it is only shape differences that are being measured. For ellipses that are close to the round sphere to which they are being compared, the $d_{sd}$ distance shown in (C) is more sensitive to small changes than the sphericity shown in (B). Computational issues arise for
the meshes used here when  $A$ becomes close to 10, as seen in (C).}
\label{fig:ellipse}
\end{figure}

\subsection{The effect of noise}

In reconstructing surfaces from scanned data, one often encounters errors in the location of vertices on a surface.
These variations of vertex positions are local in nature and do not affect the overall shape of a surface, but can
cause crinkling and spiking effects locally.
To measure the effect on such noise related local deformations of a surface, we added 
Gaussian noise to the surface of a sphere and measured the $d_{sd}$ distance of the resulting surface from a round sphere. 
The  mesh used to represent the unit sphere had mean length $0.008$ and we added Gaussian random radial noise to each vertex, with standard deviation equal to a multiple $N$ of the average mesh edge length. The results, shown in Figure ~\ref{fig:noise} are extremely promising. They indicate that a random perturbation whose standard deviation is between zero and  the average edge length of the mesh
is recognized by $d_{sd}$  as being close to a round sphere.
     
\begin{figure}[htbp] 
\centering
\includegraphics[width=5in]{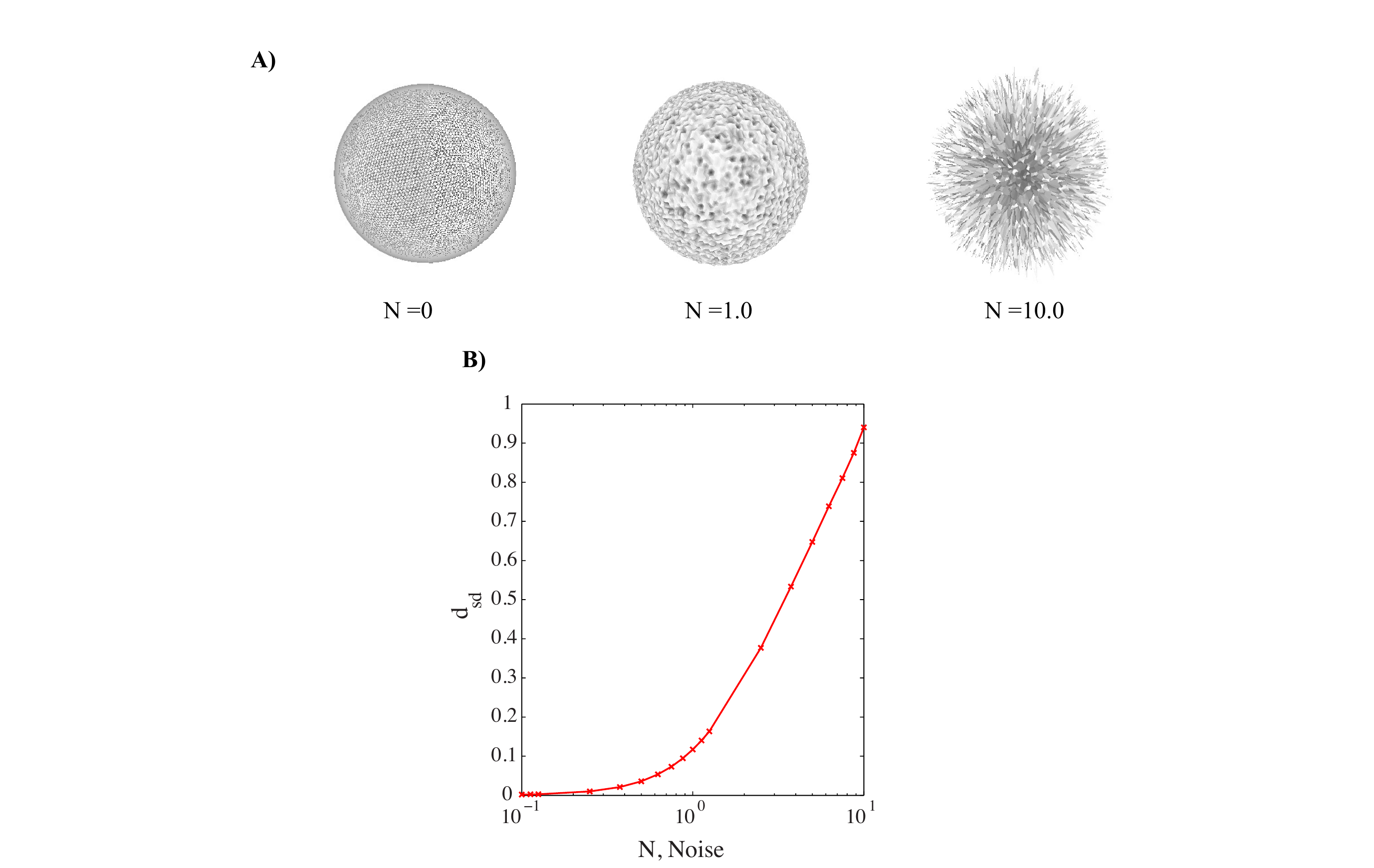} 
\caption{The effect of noise, or local deformation of the geometry, on the distance between two surfaces is indicated by adding Gaussian random noise to the vertices of a unit sphere.  Again all surfaces are scaled to have area one. $N$ indicates the standard deviation of the Gaussian as a multiple of the mean edge length, equal to 0.008 in this example.}
\label{fig:noise}
\end{figure}

\subsection{Subdividing a mesh}
    
\begin{figure}[htbp] 
\centering
\includegraphics[width=2in]{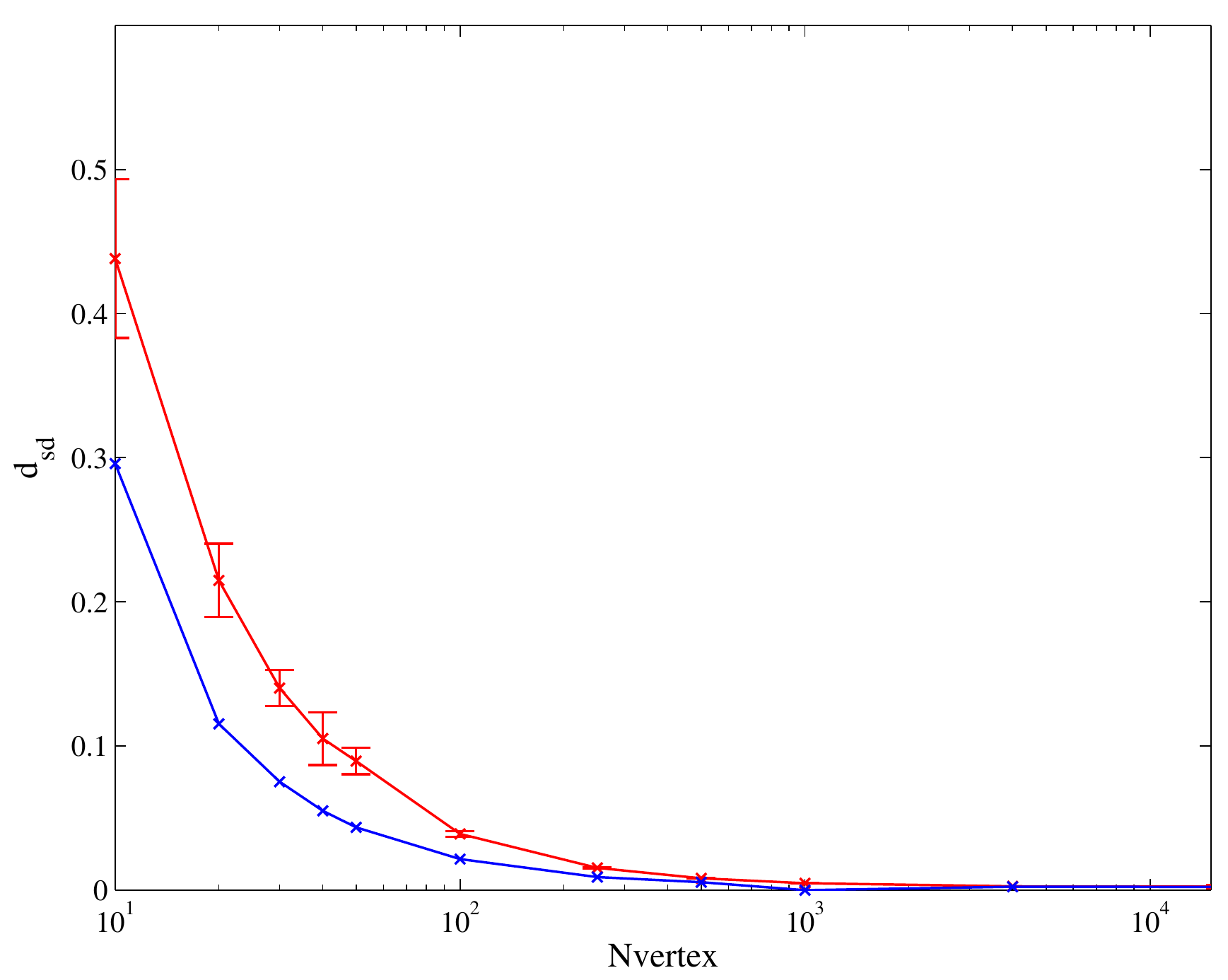} 
\caption{Distances of spheres represented with varying numbers of mesh points from the unit sphere with 1000 uniformly distributed vertices. The upper curve gives the $d_{sd}$ distance as a function of the numbers of  randomly distributed vertices. The lower curve gives the $d_{sd}$ distance as a function of the number of uniformly distributed points. Note that as the number of uniformly distributed points drops from 10,000 to 1,000, the $d_{sd}$-distance remains near zero, indicating
that a change in mesh does not affect the distance when there are enough points to accurately model the surface.  For low number of vertices, the graphs reflect a larger deviation from roundness of the surfaces represented by the mesh.}
\label{fig:vertices}
\end{figure}

To understand the effect of the choice of mesh on $d_{sd}$, we experimented with the effects of
 simplifying a mesh by removing points.
We take for our first surface $F_1$ a sphere $S_1$ whose surface is discretized with 1000 points, distributed uniformly on the surface.
We compare $F_1$ with a series of spheres having varying numbers of mesh points. All surfaces are scaled to have area one

In  Experiment 1 a second sphere is represented with  $N$ vertex points on its surface that are placed randomly, for values of $N$  up to $N =1000$.  Larger values of $N$, up to $N= 16000$, are obtained by subdividing each of the 1000 triangles into either four or 16 similar triangles.
The experiment  is repeated 50 times for each value of $N$ and  the average and standard deviation of $d_{sd}$ is obtained for these 50 samples.
The number of vertices $N$ varies from 10 to 16000, indicated by the upper plot in Figure~\ref{fig:vertices}.
In Experiment 2 the second sphere is again represented with $N$ points on its surface with $N$ between 10 and 16000, but this time the positions of the points are optimized to give a distribution that is as uniform as possible. The resulting distances are shown by the lower plot on the Figure~\ref{fig:vertices}.
The results in both cases indicate that a change of mesh does not affect the $d_{sd}$-distance as long as enough vertices 
are kept to maintain a close approximation of the underlying geometrical surface.  For uniformly distributed points, the number of points required to densely approximate the surface of the sphere is smaller than for randomly distributed points, causing a gap between the two graphs. The location of the vertices of the meshes has no effect once there are enough to
accurately capture the geometry of the round sphere.

\subsection{Chirality and Reflections}

The $d_{sd}$-distance measures the  symmetric distortion energy of an orientation preserving
diffeomorphism.  This distance can be reduced significantly 
if we also allow orientation reversing diffeomorphisms. For example, comparing a right hand and mirror-image left hand with $d_{sd}$  will give a non-zero distance.  There are circumstances when we want to ignore this difference
in orientation, or chirality.

We can specify that we wish to incorporate into our shape analysis either only orientation preserving diffeomorphisms, or alternately both orientation preserving and orientation reversing 
diffeomorphisms.  To allow for orientation reversing correspondences, when comparing a
surface $F_1$ to $F_2$, we  add an additional surface $\bar F_1$ which we also compare to $F_2$.  The surface 
 $\bar F_1$ is obtained by reflecting  $F_1$, computed by multiplying the $z$-coordinate of each vertex  of $F_1$ by $-1$.
This gives twice as many candidates for an $E_{sd}$ minimizing map, and may lead to a smaller distance.
We denote the distance of two surfaces given by minimizing in this larger class of 
potential correspondences by $\bar d_{sd}$, so that 
$\bar d_{sd}(F_1,F_2)  = \min \{ d_{sd}(F_1,F_2)  , d_{sd}(\bar F_1,F_2)  \} $.

To see the effect of adding orientation reversing diffeomorphisms, 
we model a right hand by a surface $F_1 $ which is a sphere with three protrusions, in the 
direction of the $ \vec i , \vec j $ and $ \vec k$ vectors.  For $F_2$ we take a sequence of surfaces where
the protrusion in the direction of   $\vec j $ is rotated  in the $xy$-plane through  $- \vec i $ to  $- \vec j $.
Its final position represents a surface isometric to the reflection of $F_1$.  We compute the distances
$d_{sd}$ from $F_1$ to each surface in this family, and then the distances
$\bar d_{sd}$ from $F_1$ which allow for orientation reversal. The results are shown in Figure~\ref{threebumps}.
All surfaces are scaled to have area one.

\begin{figure}[htbp] 
\centering
\includegraphics[width=4in]{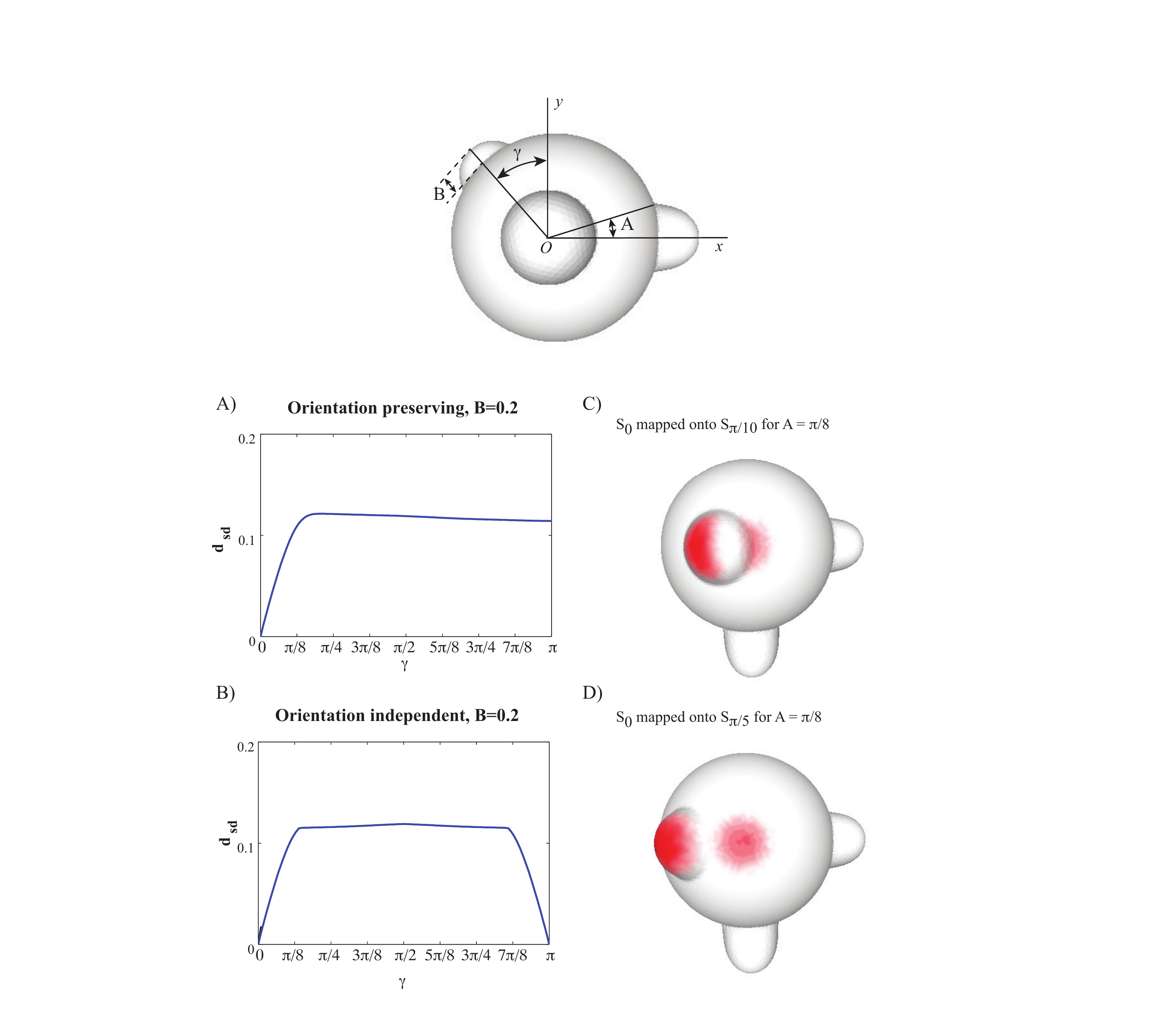} 
\caption{Distances of spheres with three protruding bumps of heights 1.2 (along the positive $y$-axis), 1.4 (along  the positive $x$-axis), and 1.6 (along  the positive $z$-axis).
The bump of height 1.2, initially facing straight out along the $y$-axis, is rotated through the
negative $x$-axis around to the negative $y$-axis. In (A) the resulting $d_{sd}$-distances are graphed with only orientation preserving diffeomorphisms allowed. The 
$d_{sd}$-distance increases initially, but then drops slightly as the bump keeps rotating to the opposite side of the sphere.
In (B) orientation reversing diffeomorphisms are also allowed, and the $\bar d_{sd}$-distance drops down to 0 after a rotation of $\pi$.
In (C) and (D), shaded red areas indicate areas of larger stretching or compression of the domain, shown on the image surface.
}
\label{threebumps}
\end{figure}

An interesting example of this phenomenon occurred in an analysis of  a collection of  teeth taken from a
variety of primates, both simians and prosimians. The high effectiveness of $d_{sd} $ in measuring similarities and differences 
between such biological shapes is described in \cite{KoehlHassScience}.
A typical set of $d_{sd}$ distances between teeth from the same and from  different families,
 is shown in Figure~\ref{examples}.
  
 \begin{figure}[htbp] 
\centering
\includegraphics[width=3in]{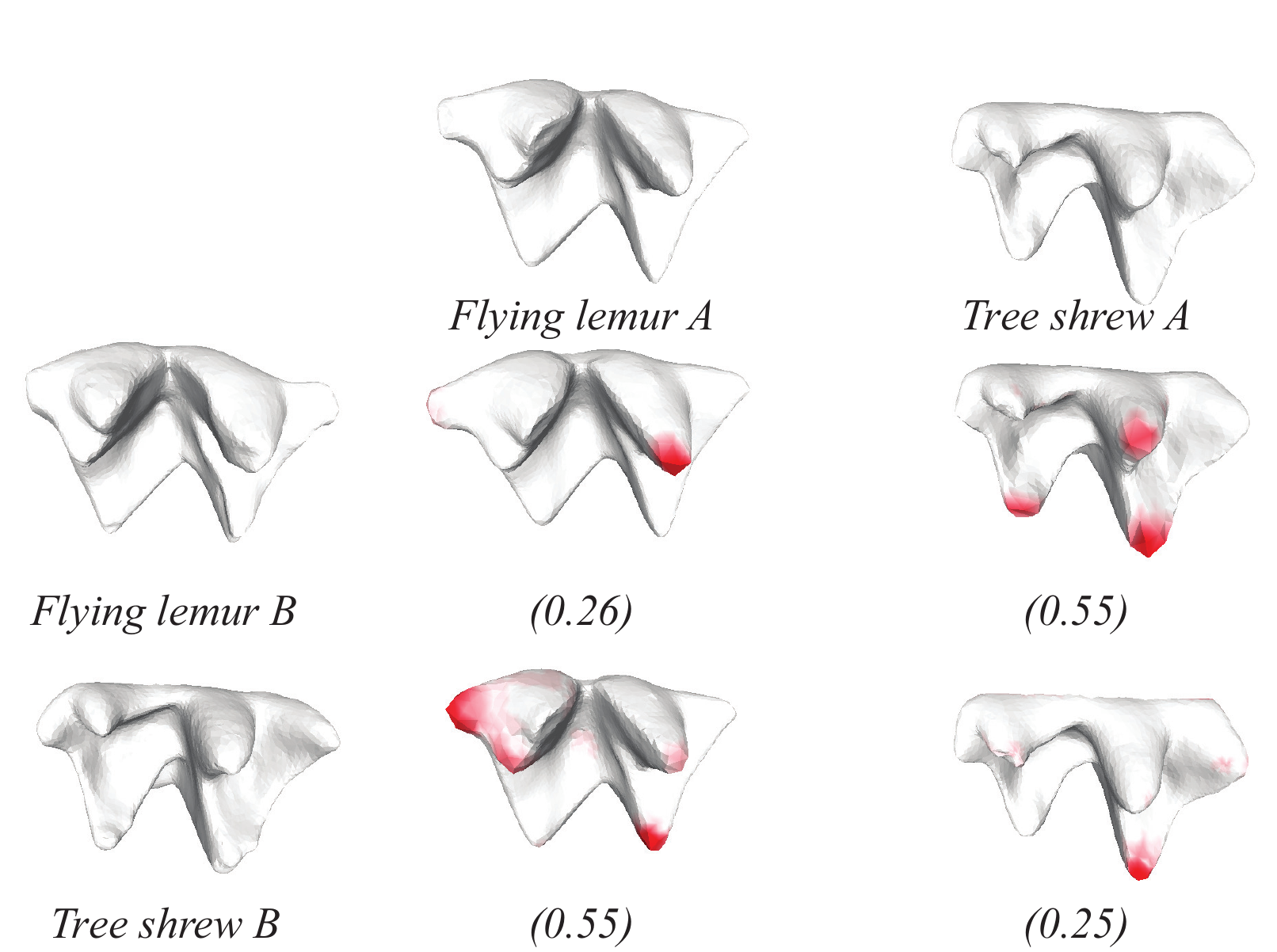} 
\caption{A computation of $d_{sd}$-distances between four teeth.  Two of the teeth are from flying lemurs and two are from tree shrews. All surfaces are scaled to have area one. The distances between teeth from the same families are lower than between teeth of different families.  The shaded red areas indicate areas where there is large stretching or compression on the domain (surface A), indicated on the image (surface B). Note that distances between pairs of teeth from the same family are smaller.}
\label{examples}
\end{figure}

Data describing the geometry of a collection of teeth was obtained from the study of
 \cite{Boyer}  and we are grateful to Y. Lipman for making it available to us.
The data contained both left and right teeth. 
 The distance between two teeth can be highly affected by the choice of whether to
 allowing orientation reversing correspondences, as indicated in Figure~\ref{orientedteeth}.
 
 \begin{figure}[htbp] 
\centering
\includegraphics[width=3in]{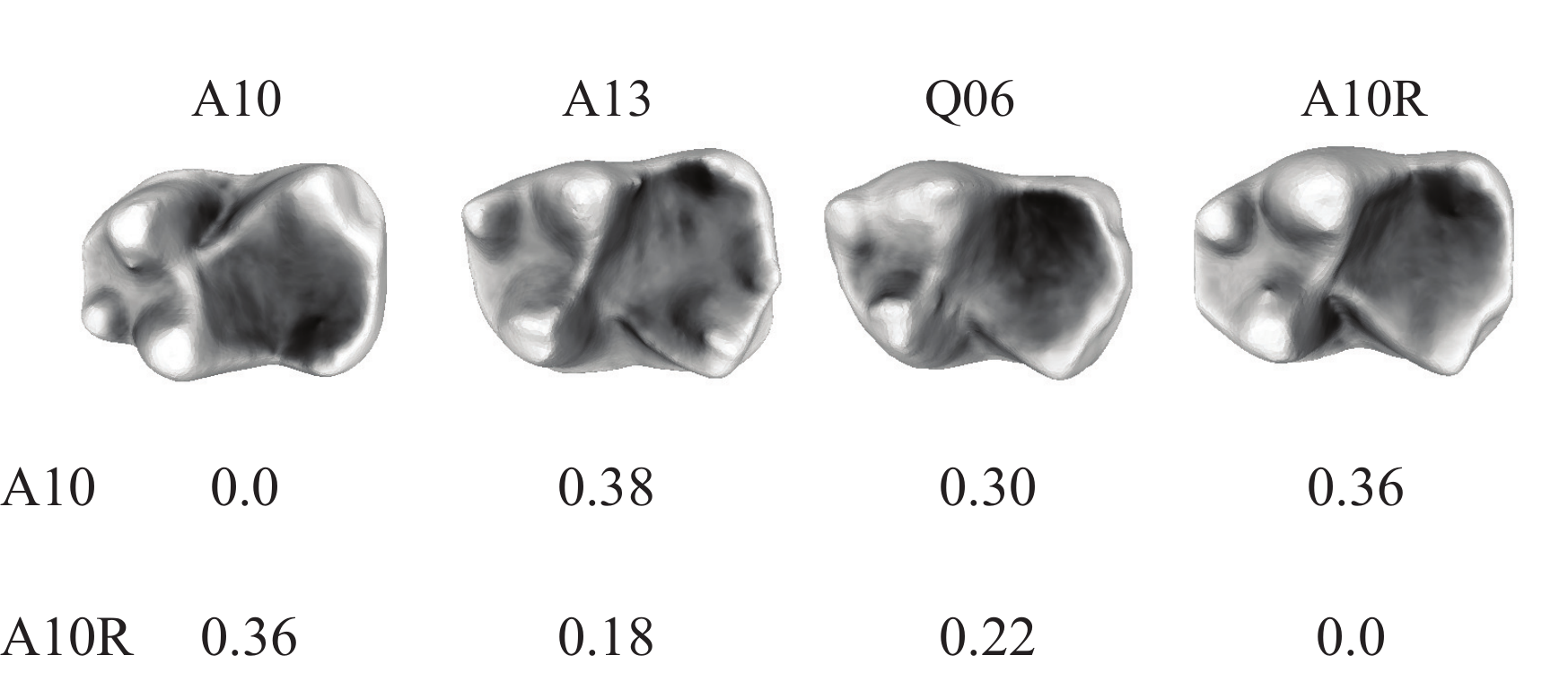} 
\caption{All teeth here belong to  euprimates, with A10 and A13 in one family and Q06 in a second.  
The $d_{sd}$-distance between molar A10 and molar A13 is 0.38.
Despite belonging to different families, the distance between  molar A10 and molar Q06 has the smaller value of 0.30.  This seeming mismatch is resolved by  considering chirality. 
While  A10 and  A13 are from the same family, they have different orientation, or handedness, while A10 and Q06 share the same orientation. 
The reflected tooth A10R has smaller $d_{sd}$-distance to A13 (0.18)  then to Q06 (0.22), indicating that $d_{sd}$-distance is capturing information
about the family to which the tooth belongs.}
\label{orientedteeth}
\end{figure}
 
When only orientable alignments were allowed, the $d_{sd}$-distance was not as effective
as either a human observer or as the continuous Procrustes distance described in  \cite{Boyer}  
at discriminating between the teeth of simians and prosimians.
The effect of allowing both orientation preserving and reversing maps is seen in the ROC analysis  in Figure ~\ref{teeth}.
In this statistical test, the effectiveness of a distance at predicting membership in a common family is given by
the area under a curve, with greater area indicating higher effectiveness. The $d_{sd}$-distance measured only
with orientation preserving alignments was not as effective as other methods at correctly identifying teeth from
the same family (dashed red curve in Figure ~\ref{teeth}).  
This occurred because both left and right molars were included in the data set. When
orientation reversing diffeomorphisms were also allowed, the $\bar d_{sd}$-distance  performed as well as the
other methods (solid red curve in Figure ~\ref{teeth}).
The results indicate that geometric differences between left and right molars within the same family 
can be larger than those between right molars from two different families.

\begin{figure}[htbp] 
\centering
\includegraphics[width=3in]{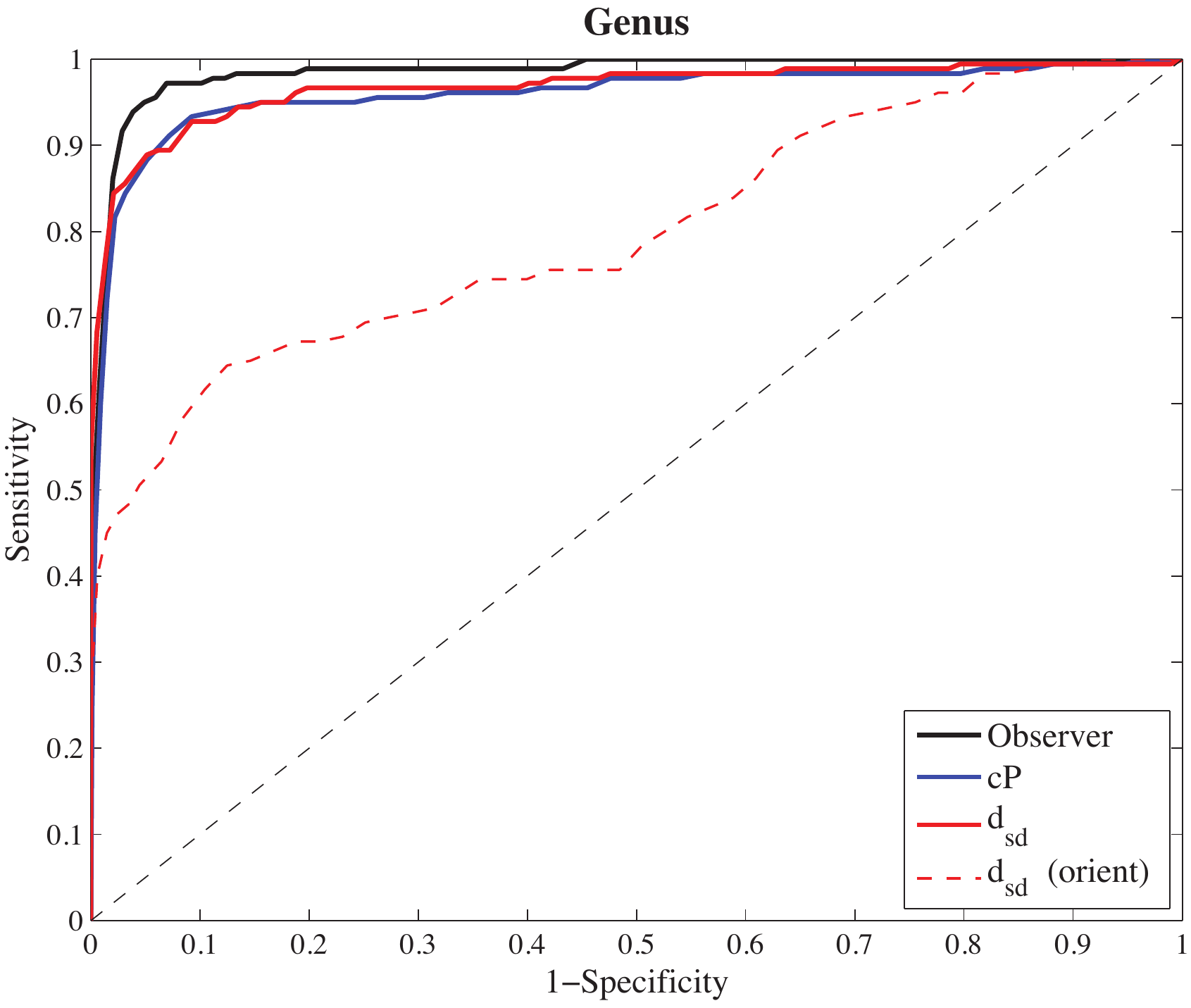} 
\caption{A statistical test of the effectiveness of a distance in identifying members of two subgroups is given by the area under a curve in a ROC analysis. Greater area indicates more effectiveness. The solid red curve results from $\bar d_{sd},$ with orientation reversing diffeomorphisms allowed in comparing shapes, while the dashed red curve restricts the computation of $d_{sd},$ to orientation preserving maps. The dashed curve resulted because $d_{sd}$  distinguished  left and right molars from the same family
that were highly similar after reflection.}
\label{teeth}
\end{figure}

\section{Conclusions}
We have described a new method of comparing the shapes of two Riemannian  surfaces
of genus zero. 
We introduced the notion of symmetric distortion energy
and established the existence of a conformal diffeomorphism between any
pair of genus-zero surfaces  that 
minimizes this energy among all conformal maps. We then established that the value of the
symmetric distortion energy on the minimizing map leads to a metric on the
space of shapes. We described how to implement this method and there results
of experiments performed  with such an implementation.
These experiments indicate that the symmetric distortion energy has
properties that are highly desirable for many classes of applications.

 \end{document}